\documentclass[11pt,reqno]{amsart}

\usepackage[T1]{fontenc}
\usepackage{lmodern}
\usepackage{microtype}
\usepackage{amsmath,amssymb,mathtools}
\usepackage{booktabs}
\usepackage{enumitem}
\usepackage{xcolor}
\usepackage{tikz}
\usepackage{aliascnt}
\usepackage[margin=1.15in]{geometry}
\usepackage[colorlinks=true,linkcolor=blue!55!black,citecolor=green!40!black,
  urlcolor=blue!60!black]{hyperref}
\usepackage[nameinlink,capitalise,noabbrev]{cleveref}

\newtheorem{theorem}{Theorem}[section]
\newaliascnt{proposition}{theorem}
\newtheorem{proposition}[proposition]{Proposition}
\aliascntresetthe{proposition}
\newaliascnt{lemma}{theorem}
\newtheorem{lemma}[lemma]{Lemma}
\aliascntresetthe{lemma}
\newaliascnt{corollary}{theorem}

\aliascntresetthe{corollary}
\theoremstyle{remark}
\newaliascnt{remark}{theorem}
\newtheorem{remark}[remark]{Remark}
\aliascntresetthe{remark}
\newaliascnt{example}{theorem}

\aliascntresetthe{example}
\theoremstyle{definition}
\newaliascnt{definition}{theorem}
\newtheorem{definition}[definition]{Definition}
\aliascntresetthe{definition}

\newcommand{\F}{\mathbb F}
\newcommand{\C}{\mathbb C}
\newcommand{\Z}{\mathbb Z}
\newcommand{\Tr}{\operatorname{Tr}}
\newcommand{\tr}{\operatorname{tr}}
\newcommand{\Frob}{\operatorname{Frob}}
\newcommand{\Fr}{\operatorname{Fr}}
\newcommand{\Dih}{\operatorname{Dih}}
\newcommand{\Cay}{\operatorname{Cay}}

\newcommand{\PG}{\operatorname{PG}}
\newcommand{\ol}[1]{\overline{#1}}

\title[Character sums and Ramanujan double covers]{Character sums on an oriented Singer conic and explicit Ramanujan double covers}
\author{Pin-Chi Hung}
\address{Department of Mathematics, Soochow University, Taipei City, Taiwan}
\email{pinchihung1111@gmail.com}

\author{Ming-Hsuan Kang}
\address{Department of Applied Mathematics, National Yang Ming Chiao Tung University, Hsinchu, Taiwan}
\email{kmsming@gmail.com}

\date{}
\hypersetup{
  pdftitle={Character sums on an oriented Singer conic and explicit Ramanujan double covers},
  pdfauthor={Pin-Chi Hung and Ming-Hsuan Kang},
  pdfsubject={Finite-field character sums and explicit Ramanujan double covers},
  pdfkeywords={Ramanujan graph, graph cover, Cayley graph, Singer difference set,
    finite-field character sum, quaternion group}
}


\begin{document}

\begin{abstract}
Let $q$ be odd.  The trace conic in $\F_{q^3}$ determines a Singer
difference set in $\F_{q^3}^\times/\F_q^\times$ and a natural
square-class lift to $\F_{q^3}^\times/\F_q^{\times2}$.  We study the odd
multiplicative Fourier coefficients of this lift and prove that they are
bounded in absolute value by $2\sqrt q$.  The proof improves the naive
six-puncture Weil bound by exploiting a projective Klein-four symmetry of
the associated rank-one local system.  The resulting nontrivial cocycle
produces a quaternionic action on its four-dimensional cohomology, while
Frobenius symmetry reduces the relevant trace to two Weil-scale
eigenvalues.

As an application, the oriented conic yields an explicit Singer-invariant
signing of the point--line incidence graph of $\PG(2,q)$.  The corresponding
dihedral Cayley graph is a connected Ramanujan double cover.  Thus a conic
lift already known in finite-geometric constructions has an additional
Ramanujan spectral property governed by its odd multiplicative character
sums.
\end{abstract}

\maketitle

\section{Introduction}

The point--line incidence graph of $\PG(2,q)$ is a classical Ramanujan
graph.  Our goal is to construct, for every odd prime power $q$, an explicit
double cover that is again Ramanujan and remains a Cayley graph.  The
construction comes from a natural lift of a trace conic over $\F_q$, and
the Ramanujan property reduces to a multiplicative character-sum estimate.
The main arithmetic problem is to prove the bound $2\sqrt q$ for the
relevant character sums.

We now describe the finite-field construction.  Let
\[
  F=\F_q,\qquad E=\F_{q^3},\qquad N=q^2+q+1,
\]
where $q$ is an odd prime power, and write $\Tr=\Tr_{E/F}$.  The quadratic
form
\[
  Q(x)=\Tr(x^2)
\]
defines a nonsingular conic $\mathcal C$ in $\PG(2,q)$.  In Singer coordinates, its
point set is the difference set
\[
  D_Q=\{\bar x\in E^\times/F^\times: \Tr(x^2)=0\}.
\]
The natural quotient
\[
  E^\times/F^{\times2}\longrightarrow E^\times/F^\times,
  \qquad F^{\times2}:=\{a^2:a\in F^\times\},
\]
is two-to-one.
Fixing an isotropic vector $e$, define a square-class lift of the conic by
\begin{equation}\label{eq:intro-orientation}
  \ell_e([x])=
  \begin{cases}
    F^{\times2}x\Tr(ex),&[x]\ne[e],\\[1mm]
    F^{\times2}(2e),&[x]=[e].
  \end{cases}
\end{equation}
Changing $e$ changes this section, at most, by one global nonsquare.

In exponent notation, the lift in \eqref{eq:intro-orientation} is the
conic set introduced by Feng, Momihara, and Xiang in their construction of
Cameron--Liebler line classes \cite[Section~3]{FengMomiharaXiang}, and
revisited by Momihara and Xiang in the study of cyclic arcs and strongly
regular Cayley graphs \cite[Section~2.3]{MomiharaXiang2022}.  Those works
establish, among other properties, additive-character identities for the
lift.  Here we instead study the part of its \emph{multiplicative} Fourier
spectrum that is odd with respect to the square-class double cover.

Write $\mathcal C^+=\ell_e(\mathcal C)$ for either coherent orientation of
the conic.  If
\[
  \lambda:E^\times\longrightarrow\C^\times,
  \qquad \lambda|_{F^\times}=\eta_F,
\]
then $\lambda$ is trivial on $F^{\times2}$ and hence is well defined on
$\mathcal C^+$.  Put
\begin{equation}\label{eq:intro-S-lambda}
  S_\lambda=\sum_{\xi\in\mathcal C^+}\lambda(\xi).
\end{equation}
Our main arithmetic result is the following Ramanujan-scale square-root estimate.

\begin{theorem}[Oriented-conic period bound]\label{thm:period-bound}
For every multiplicative character $\lambda:E^\times\to\C^\times$ with
$\lambda|_{F^\times}=\eta_F$, one has
\[
  |S_\lambda|\le2\sqrt q.
\]
\end{theorem}

The constant $2$ is the essential point. After the conic is split, the
character sum is naturally associated with a rank-one system having six
geometric singularities. A direct Weil bound therefore gives only
$4\sqrt q$. The special symmetry of the six singularities provides an
additional reduction, cutting the effective contribution in half and
yielding the required bound $2\sqrt q$. The mechanism behind this
improvement is developed in Section~\ref{sec:six-point-estimate}.

This is exactly the bound required for the Ramanujan property on the
graph side. Recall that a connected $k$-regular graph is Ramanujan if every
adjacency eigenvalue other than $k$ (and also $-k$ in the bipartite case)
has absolute value at most $2\sqrt{k-1}$. In particular, the Ramanujan
bound for a $(q+1)$-regular graph is $2\sqrt q$.

Let $\mathcal I_q$ be the point--line incidence graph of $\PG(2,q)$.
It is a $(q+1)$-regular bipartite Ramanujan graph, and the Singer action
identifies it with the dihedral Cayley graph
\[
  \mathcal I_q\cong
  \Cay\!\left(\Dih(C_N),\{sd:d\in D_Q\}\right),
  \qquad C_N=E^\times/F^\times.
\]
Put $C_{2N}=E^\times/F^{\times2}$ and let
$\ell_Q:D_Q\to C_{2N}$ be the lift in
Definition~\ref{def:lifted-D}. The odd characters of $C_{2N}$ are exactly
the characters occurring in \Cref{thm:period-bound}; their periods are
therefore the new eigenvalues of the lifted Cayley graph.

\begin{theorem}[Explicit Ramanujan double cover]\label{thm:main}
For every odd prime power $q$, the graph
\begin{equation}\label{eq:main-graph}
  \widetilde{\mathcal I}_q
  =
  \Cay\!\left(
    \Dih(C_{2N}),
    \{s\ell_Q(d):d\in D_Q\}
  \right)
\end{equation}
is a connected $(q+1)$-regular bipartite Ramanujan graph.  The quotient
$C_{2N}\to C_N$ induces a graph double cover
\[
  \widetilde{\mathcal I}_q\longrightarrow\mathcal I_q.
\]
Equivalently, $\ell_Q$ is an explicit Singer-invariant Ramanujan signing
of the incidence graph.
\end{theorem}

Ramanujan graphs obtained from finite fields are often proved optimal by
turning their nontrivial eigenvalues into character sums and applying the
Riemann hypothesis for curves; see, for example,
\cite{FengLi,LiMeemark}.  The present construction follows this general
philosophy, but the character sum in \Cref{thm:period-bound} is not
controlled by the number of punctures alone.  The additional quaternionic
symmetry is what lowers the natural four-eigenvalue estimate to the
Ramanujan threshold.

From the graph-cover viewpoint, a signing of a graph determines a double
cover, and the new spectrum is the spectrum of the signed adjacency
matrix \cite{BiluLinial}.  Marcus, Spielman, and Srivastava proved that
every finite regular bipartite graph has a Ramanujan double cover
\cite{MSS}.  Their theorem is existential with respect to the geometry of
the present incidence graph.  Here the signing is given by the explicit
finite-field formula \eqref{eq:intro-orientation}, is constant on Singer
edge orbits, and the resulting cover is again a Cayley graph.

The construction is intrinsically tied to odd characteristic.  In
Section~\ref{sec:characteristic-two} we explain why the square-class cover
and the nonsingular trace conic both degenerate in characteristic two.
The remaining sections develop the Singer model and the intrinsic
orientation, reduce the new spectrum to \eqref{eq:intro-S-lambda}, prove
the six-point quaternion estimate, and apply it to the conic periods.

\section{Singer symmetry and Ramanujan sections}
\label{sec:singer-lift}

\subsection{Ramanujan sections of dihedral Cayley graphs}

For a finite cyclic group $A$, write
\[
  \Dih(A)=A\rtimes\langle s\rangle,
  \qquad
  s^2=1,
  \qquad
  sas=a^{-1}\quad(a\in A).
\]
Let $\Cay(H,S)$ denote the Cayley graph
with vertex set $H$ and edges $g\sim gt$ for $t\in S=S^{-1}$.

Suppose a bipartite Cayley graph is written as
\[
  G=\Cay\!\left(\Dih(C),\{sd:d\in D\}\right),
\]
where $C$ is cyclic.  Left multiplication by $C$ partitions the edges
into the orbits indexed by $d\in D$.  Hence a signing compatible with
this cyclic symmetry is determined by one sign for each generator $sd$,
rather than one sign for each edge.

Such signings can be encoded group-theoretically.  Let
\begin{equation}\label{eq:abstract-twofold-extension}
  1\longrightarrow\{1,\varepsilon\}
  \longrightarrow\widetilde C
  \overset{\pi}{\longrightarrow}C
  \longrightarrow1
\end{equation}
be a twofold extension of finite cyclic groups.  Choosing the sign of
$sd$ is equivalent, after fixing a reference lift, to choosing one of
the two elements of $\pi^{-1}(d)$.  This motivates the following
terminology.

\begin{definition}\label{def:ramanujan-section}
A \emph{section of $D$ through $\pi$} is a map
\[
  \ell:D\longrightarrow\widetilde C
  \qquad\text{such that}\qquad
  \pi(\ell(d))=d\quad(d\in D).
\]
Its image $\widetilde D=\ell(D)$ contains exactly one element above each
element of $D$.  The section determines the lifted dihedral Cayley graph
\begin{equation}\label{eq:abstract-section-graph}
  \widetilde G_\ell
  =
  \Cay\!\left(
    \Dih(\widetilde C),
    \{s\ell(d):d\in D\}
  \right).
\end{equation}
We call $\ell$ a \emph{Ramanujan section} if
$\widetilde G_\ell$ is connected and Ramanujan.
\end{definition}

\subsection{The Singer model of the incidence graph}

Put
\[
  F=\F_q,\qquad E=\F_{q^3},\qquad N=q^2+q+1,
\]
and write $\Tr=\Tr_{E/F}$.  Define
\[
  C_N=E^\times/F^\times,
  \qquad
  C_{2N}=E^\times/F^{\times2}.
\]
For $x\in E^\times$, use the quotient-class notation
\[
  \bar x=xF^\times\in C_N,
  \qquad
  \widetilde x=xF^{\times2}\in C_{2N}.
\]
The natural quotient is part of an exact sequence
\begin{equation}\label{eq:cyclic-exact-sequence}
  1\longrightarrow F^\times/F^{\times2}
  \longrightarrow C_{2N}
  \overset{\pi}{\longrightarrow} C_N
  \longrightarrow1.
\end{equation}
Since $q$ is odd, its kernel has order two.  Thus $|C_N|=N$ and
$|C_{2N}|=2N$.

Since $N$ is odd, squaring is an automorphism of $C_N$.  Label the points
and lines of $\PG(2,q)$ by
\[
  P_{\bar x}=Fx^2
\]
and
\[
  L_{\bar y}
  =
  \left\{
    Fz\in\PG(2,q):
    \Tr(y^2z)=0
  \right\}.
\]
Replacing $x$ or $y$ by an $F^\times$-multiple does not change the
corresponding point or line, so the labels are well defined.  They are
bijections: squaring permutes $C_N$, and the trace pairing on $E/F$ is
nondegenerate.

Define the conic Singer difference set
\[
  D_Q
  =
  \left\{
    \bar x\in C_N:
    \Tr(x^2)=0
  \right\}.
\]
This is well defined because replacing $x$ by $ax$, $a\in F^\times$,
multiplies $\Tr(x^2)$ by $a^2$.  The incidence relation is
\begin{equation}\label{eq:singer-incidence}
  P_{\bar x}\in L_{\bar y}
  \quad\Longleftrightarrow\quad
  \Tr((xy)^2)=0
  \quad\Longleftrightarrow\quad
  \bar x\,\bar y\in D_Q.
\end{equation}

Identify the two vertex classes with the two cosets of $C_N$ in
$\Dih(C_N)$ by
\[
  P_c\longleftrightarrow c,
  \qquad
  L_c\longleftrightarrow sc.
\]
For $c,d\in C_N$, the Cayley edge defined by the reflection $sd$ sends
$c$ to
\[
  c(sd)=s(c^{-1}d).
\]
The product of the point label $c$ and the line label $c^{-1}d$ is $d$.
Thus \eqref{eq:singer-incidence} gives
\begin{equation}\label{eq:incidence-cayley}
  \mathcal I_q
  \cong
  \Cay\!\left(
    \Dih(C_N),
    \{sd:d\in D_Q\}
  \right).
\end{equation}

Thus the symmetry-restricted problem for $\mathcal I_q$ is concrete:
find a Ramanujan section of
\[
  D_Q\subseteq C_N
\]
through the natural twofold quotient
\[
  \pi:C_{2N}\longrightarrow C_N.
\]
The remainder of this section constructs the candidate section; the next
section explains its geometric meaning.

The set $D_Q$ is a Singer $(N,q+1,1)$ difference set
\cite{Singer}.  In particular, it has $q+1$ elements; the preceding
Cayley realization is the classical difference-set development of the
projective plane.

The following proposition records the standard spectrum of the incidence
graph of a projective plane, viewed as the incidence graph of a symmetric
design; see \cite{Dembowski,BrouwerHaemers}.

\begin{proposition}[Classical incidence spectrum]\label{prop:base-cayley}
The Cayley graph in \eqref{eq:incidence-cayley} has spectrum
\[
  \{q+1,-(q+1),
    (\sqrt q)^{[N-1]},(-\sqrt q)^{[N-1]}\}.
\]
\end{proposition}

\begin{proof}
If $M$ is the point--line incidence matrix, then
$MM^{\mathsf T}=qI+J$, where $I$ is the $N\times N$ identity matrix and
$J$ is the $N\times N$ all-ones matrix.  Hence the singular values of
$M$ are $q+1$ once and $\sqrt q$ with multiplicity $N-1$, which gives
the displayed spectrum.
\end{proof}

\subsection{The explicit conic section}

Consider the quadratic form and trace pairing
\[
  Q(x)=\Tr(x^2),
  \qquad
  B(x,y)=\Tr(xy)
\]
on the three-dimensional $F$-space $E$.  The set $D_Q$ is the projective
conic $Q=0$ written inside $C_N=E^\times/F^\times$.

Fix $d_0=\bar x_0\in D_Q$ together with a representative
$x_0\in E^\times$.  The
tangent line at $Fx_0$ is
\[
  B(x_0,z)=0.
\]
Since a tangent to a nonsingular conic contains no other conic point, for a second conic point $d=\bar x\ne d_0$,
\[
  B(x_0,x) \ne0.
\]
The square class of this tangent evaluation decides which of the two
lifts of $d$ is selected.

\begin{definition}\label{def:lifted-D}
For $d=\bar x\in D_Q\setminus\{d_0\}$, define
\[
  \ell_Q(d)
  =
  \widetilde{\,xB(x_0,x)\,}
  \in C_{2N}.
\]
At the base point, put
\[
  \ell_Q(d_0)=\widetilde{2x_0}.
\]
Finally, set
\begin{equation}\label{eq:explicit-lift-intro}
  \widetilde D
  =
  \ell_Q(D_Q)
  \subseteq C_{2N},
  \qquad
  s\widetilde D=\{s\widetilde d:\widetilde d\in\widetilde D\}.
\end{equation}
\end{definition}

The formula is well defined and gives a section of $D_Q$ through
\eqref{eq:cyclic-exact-sequence}; its intrinsic verification, including
the exceptional value at $d_0$, is given in
Section~\ref{sec:oriented-conic}.  In primitive-element coordinates this is
the conic lift $X$ of Feng--Momihara--Xiang
\cite[Eq.~(3.4)]{FengMomiharaXiang}, also denoted $X_Q$ in
\cite[Eq.~(2.14)]{MomiharaXiang2022}; the orientation language below gives
an intrinsic formulation independent of a choice of primitive element.

The next proposition is the routine group-theoretic translation from a
section of the generator set to a graph $2$-cover.

\begin{proposition}[Section-to-cover correspondence]\label{prop:lift-cover}
The quotient homomorphism
\[
  \Dih(C_{2N})\longrightarrow\Dih(C_N),
  \qquad c\longmapsto\pi(c),\quad s\longmapsto s,
\]
induces a graph double cover
\[
  \Cay\!\left(
    \Dih(C_{2N}), s\widetilde D
  \right)
  \longrightarrow
  \mathcal I_q.
\]
For each $d\in D_Q$, choosing $\widetilde d\in\pi^{-1}(d)$ is equivalent
to choosing a sign on the Singer orbit of incidence edges indexed by $d$.
\end{proposition}

\begin{proof}
The group homomorphism has the order-two kernel from
\eqref{eq:cyclic-exact-sequence}, and it maps the displayed connection
set bijectively to $\{sd:d\in D_Q\}$.  Therefore the two vertices above
any base vertex have neighborhoods mapping bijectively to its
neighborhood, which is precisely the graph-cover condition.  Under left
multiplication by $C_N$, the incidence edges with fixed product $d$ form
one Singer orbit.  Choosing either element of $\pi^{-1}(d)$ determines
whether that orbit joins equal or opposite sheets, hence determines its
sign.
\end{proof}

\section{The oriented conic}
\label{sec:oriented-conic}

The preceding section reduced a Singer-invariant signing to the choice of
one of the two square-class lifts above each conic point.  An arbitrary
choice gives an arbitrary sign map on the Cayley generators.  The purpose
of this section is to isolate the geometric choices: there are exactly two
globally coherent sections, and the section $\ell_Q$ constructed above is
one of them.

On the three-dimensional $F$-space $E$, put
\[
  Q(x)=\Tr(x^2),
  \qquad
  B(x,y)=\Tr(xy).
\]
Write
\[
  \mathcal C
  =
  \{F^\times x\in\mathbb P(E):Q(x)=0\}
\]
for the nonsingular projective conic and
\[
  \widetilde{\mathcal C}
  =
  \{F^{\times2}x:x\in E^\times,\ Q(x)=0\}
\]
for its square-class double cover.  The projection
\[
  \rho:\widetilde{\mathcal C}\longrightarrow\mathcal C,
  \qquad
  F^{\times2}x\longmapsto F^\times x,
\]
has two elements in every fiber.

The pairing $B$ provides the notion of consistency.  We call a section
\[
  \ell:\mathcal C\longrightarrow\widetilde{\mathcal C}
\]
an \emph{orientation} if the pairings of any two distinct selected lifts
have the same square class; for the trace conic, this condition is
\[
  B(x,y)\in 2F^{\times2}
\]
whenever $\ell([x])=F^{\times2}x$ and
$\ell([y])=F^{\times2}y$ with $[x]\ne[y]$.  The condition is independent of the chosen representatives.

Fix a nonzero isotropic vector $e$.  For $[x]=F^\times x\ne[e]$, the
tangent line at $[e]$ is $B(e,z)=0$, so $B(e,x)\ne0$.  The most direct
way to orient $[x]$ relative to $[e]$ is to choose
\[
  F^{\times2}xB(e,x),
\]
because
\[
  B\bigl(e,xB(e,x)\bigr)=B(e,x)^2
\]
is automatically a square.  Once the lift of the base point is chosen as
$F^{\times2}(2e)$, the corresponding pairing lies in $2F^{\times2}$.
The formula is projectively well defined:
replacing $x$ by $ax$ multiplies $xB(e,x)$ by the square $a^2$.

At the base point we use the compatible value
$F^{\times2}(2e)$; for example, the same rule based at an isotropic
$f$ normalized by $B(f,e)=2$ gives this lift.  Thus define
\begin{equation}\label{eq:intrinsic-orientation}
  \ell_e([x])
  =
  \begin{cases}
    F^{\times2}xB(e,x),
      & [x]\ne[e],\\[2mm]
    F^{\times2}(2e),
      & [x]=[e].
  \end{cases}
\end{equation}

The fact that normalization relative to one point forces consistency
between every pair is the elementary core of the construction.

We first record a property of the trace form.  If
$x_1,x_2,x_3$ is any $F$-basis of $E$, then
\begin{equation}\label{eq:trace-discriminant-square}
  \det\bigl(B(x_i,x_j)\bigr)_{1\le i,j\le3}
  \in F^{\times2}.
\end{equation}
Indeed, for the Moore matrix
\[
  \mathsf M=(x_i^{q^j})_{1\le i\le3,\ 0\le j\le2},
\]
one has
\[
  \bigl(B(x_i,x_j)\bigr)=\mathsf M\mathsf M^{\mathsf T}.
\]
Frobenius cyclically permutes the three columns of $\mathsf M$.  Since a
$3$-cycle is even, $\det\mathsf M\in F^\times$, and therefore the Gram
determinant is $(\det\mathsf M)^2$.

The next proposition packages the square-class coherence underlying the
conic lift in an intrinsic form.  Its Gram-determinant argument is closely
related to \cite[Lemmas~3.3--3.4]{FengMomiharaXiang}; the formulation as an
orientation is convenient for the spectral application below.

\begin{proposition}[The two coherent orientations]\label{prop:conic-orientations}
The map $\ell_e$ in \eqref{eq:intrinsic-orientation} is an orientation.
Every orientation is either $\ell_e$ or the section obtained by
multiplying all its values by a fixed nonsquare in $F^\times$.  Hence
$\mathcal C$ has exactly two orientations.

For $e=x_0$, the image of $\ell_e$ is the lifted Singer set
$\widetilde D$ of Definition~\ref{def:lifted-D}.
\end{proposition}

\begin{proof}
We have already checked that $\ell_e$ is a well-defined section.  Let
$[x]$ and $[y]$ be distinct conic points, neither equal to $[e]$, and put
\[
  a=B(e,x),\qquad b=B(e,y),\qquad c=B(x,y).
\]
A projective line meets a nonsingular conic in at most two points, so
$e,x,y$ form an $F$-basis of $E$.  Their Gram matrix is
\[
  \begin{pmatrix}
    0&a&b\\
    a&0&c\\
    b&c&0
  \end{pmatrix},
\]
with determinant $2abc$.  By
\eqref{eq:trace-discriminant-square},
\[
  2abc\in F^{\times2}.
\]
Consequently,
\[
  B\bigl(xB(e,x),yB(e,y)\bigr)
  =abc
  \in 2F^{\times2},
\]
since $2abc$ is a square and $2$ and $1/2$ have the same square class.
For the base point and $[x]\ne[e]$,
\[
  B\bigl(2e,xB(e,x)\bigr)
  =2B(e,x)^2
  \in 2F^{\times2}.
\]
Thus $\ell_e$ is an orientation.

Let $\ell$ be another orientation.  For each $p\in\mathcal C$, write
\[
  \ell(p)=\epsilon_p\ell_e(p),
  \qquad
  \epsilon_p\in F^\times/F^{\times2}.
\]
For distinct $p,p'$, coherence of both sections gives
$\epsilon_p\epsilon_{p'}=1$.  Since
$F^\times/F^{\times2}$ has order two, all $\epsilon_p$ are equal.
Therefore $\ell$ is either $\ell_e$ or its global nonsquare multiple.

Taking $e=x_0$ gives exactly the formula of
Definition~\ref{def:lifted-D}.
\end{proof}

For the chosen representative $e=x_0$, write
\[
  \mathcal C^+
  :=
  \ell_e(\mathcal C)
  =
  \widetilde D.
\]
If $c\in F^\times$ is nonsquare, the other orientation is
$c\mathcal C^+$.  For every character $\lambda$ with
$\lambda|_{F^\times}=\eta_F$,
\[
  \sum_{\xi\in c\mathcal C^+}\lambda(\xi)
  =\lambda(c)S_\lambda
  =-S_\lambda.
\]
Thus the two coherent orientations have the same absolute-value bound.

\subsection{The new spectrum and the period bound}
\label{subsec:new-spectrum}

The characters of $C_{2N}=E^\times/F^{\times2}$ split into two classes.
Those that are trivial on $F^\times/F^{\times2}$ factor through
$C_N=E^\times/F^\times$ and recover the old spectrum of
$\mathcal I_q$.  The remaining characters correspond precisely to
multiplicative characters
\[
  \lambda:E^\times\longrightarrow\C^\times
\]
satisfying
\begin{equation}\label{eq:restriction-lambda}
  \lambda|_{F^\times}=\eta_F,
\end{equation}
where $\eta_F$ is the quadratic character of $F^\times$.

Since $\lambda$ is trivial on $F^{\times2}$, it is well defined on the
oriented conic
\[
  \mathcal C^+=\widetilde D\subset E^\times/F^{\times2}.
\]
Define
\begin{equation}\label{eq:S-lambda}
  S_\lambda
  =
  \sum_{\xi\in\mathcal C^+}\lambda(\xi),
\end{equation}
which agrees with \eqref{eq:intro-S-lambda}.

We record explicitly the elementary Fourier reduction used below.

\begin{proposition}[Dihedral Fourier blocks]\label{prop:dihedral-fourier}
Let $A$ be a finite cyclic group and $D\subseteq A$.  For
\[
  G=\Cay\!\left(\Dih(A),\{sd:d\in D\}\right)
\]
and a character $\chi\in\widehat A$, put
\[
  S_\chi(D)=\sum_{d\in D}\chi(d).
\]
Then the adjacency operator of $G$ has, on the two-dimensional Fourier
block indexed by $\chi$, the matrix
\[
  \begin{pmatrix}
    0&S_\chi(D)\\
    \overline{S_\chi(D)}&0
  \end{pmatrix},
\]
and hence the two eigenvalues $\pm|S_\chi(D)|$.
Consequently, for $A=C_{2N}$ and $D=\widetilde D$, the characters factoring
through $C_N$ give the old spectrum, while the remaining characters give
precisely the eigenvalues $\pm|S_\lambda|$ with
$\lambda|_{F^\times}=\eta_F$.
\end{proposition}

\begin{proof}
Decompose the functions on each of the two cosets $A$ and $sA$ into
characters of $A$.  On the $\chi$-isotypic line in one color class, the
bipartite incidence operator is multiplication by $S_\chi(D)$; on the
opposite color class it is multiplication by its complex conjugate.
This gives the displayed Hermitian block.  In the extension
$C_{2N}\to C_N$, a character factors through $C_N$ exactly when it is
trivial on $F^\times/F^{\times2}$.  Otherwise its pullback to $E^\times$
is trivial on $F^{\times2}$ and nontrivial on $F^\times/F^{\times2}$,
which is equivalent to \eqref{eq:restriction-lambda}.
\end{proof}

Thus \Cref{thm:period-bound} is exactly the estimate needed for the new
spectrum.  We prove it in Section~\ref{sec:six-point-estimate}: first in an
abstract six-point form, and then for the oriented conic in
Subsection~\ref{subsec:application-oriented-conic}.

\section{The six-point estimate}
\label{sec:six-point-estimate}

Over $\ol F$, the three embeddings of $E$ give an identification
\[
  E\otimes_F\ol F\simeq \ol F^3.
\]
Let
\[
  [x_0:x_1:x_2]
\]
denote the resulting homogeneous coordinates on
$\mathbb P(E\otimes_F\ol F)\simeq\mathbb P^2_{\ol F}$.  In these
coordinates, the trace conic becomes
\begin{equation}\label{eq:split-conic}
  \mathcal C_{\ol F}:
  \qquad
  x_0^2+x_1^2+x_2^2=0.
\end{equation}

The diagonal equation is preserved by changing the sign of any one
coordinate.  Since simultaneous negation of all three coordinates is
projectively trivial, the three nontrivial coordinate sign changes form
a Klein four-group.

For each $i$, let
\begin{equation}\label{eq:coordinate-pairs}
  Z_i=\mathcal C_{\ol F}\cap\{x_i=0\}.
\end{equation}
The set $Z_i$ consists of the two fixed points of the sign change in the
$x_i$-coordinate.  Thus the three coordinate lines cut out three
distinguished pairs
\[
  Z_0,\qquad Z_1,\qquad Z_2
\]
on the conic, giving six points in total.

\subsection{A motivating Kummer model}

We now pass from the ambient coordinates
$[x_0:x_1:x_2]$ on $\mathbb P^2_{\ol F}$ to a coordinate on the conic.
Fix $\iota\in\ol F$ with $\iota^2=-1$.  The map
\[
  \phi:\mathbb P^1_{\ol F}\longrightarrow\mathcal C_{\ol F},
  \qquad
  [u:v]\longmapsto
  [u^2-v^2:\iota(u^2+v^2):2uv]
\]
is an isomorphism onto the conic
\[
  x_0^2+x_1^2+x_2^2=0.
\]
Writing $t=u/v$, one obtains
\[
  \phi^{-1}(Z_2)=\{0,\infty\},
  \qquad
  \phi^{-1}(Z_0)=\{1,-1\},
  \qquad
  \phi^{-1}(Z_1)=\{\iota,-\iota\}.
\]
These are the fixed-point pairs of
\[
  t\longmapsto -t,\qquad
  t\longmapsto \frac1t,\qquad
  t\longmapsto -\frac1t,
\]
which generate the Klein four-group on the six-punctured line.

The following Kummer model motivates the projective symmetry used in the
formal proof.  Suppose a rank-one tame local system has common monodromies
$a_2,a_0,a_1$ on the three opposite pairs above and
$a_0a_1a_2=-1$, as will occur for the oriented conic.  Choose an even
integer $m$ divisible by their orders, a primitive $m$th root $\xi_m$,
and integers $k_i$ with $a_i=\xi_m^{k_i}$.  After changing the $k_i$ by
multiples of $m$, write
\[
  k_0+k_1+k_2=\frac{\nu m}{2},
  \qquad \nu\ \text{odd}.
\]
The corresponding character occurs in the cohomology of the Kummer cover
\begin{equation}\label{eq:explicit-kummer-cover}
  y^m=t^{k_2}(t^2-1)^{k_0}(t^2+1)^{k_1}.
\end{equation}
The involutions $\mathsf a(t)=-t$ and $\mathsf b(t)=1/t$ lift as
\[
  \widetilde{\mathsf a}(t,y)=(-t,\alpha y),
  \qquad
  \widetilde{\mathsf b}(t,y)=(1/t,\beta t^{-\nu}y)
\]
for suitable constants $\alpha,\beta$.  Since $\nu$ is odd,
$t\mapsto-t$ changes the sign of $t^{-\nu}$, so their commutator is the
central deck transformation $y\mapsto-y$.  On the relevant character
component the two lifts therefore anticommute.  This is the geometric
phenomenon formalized in the lemma below.

\subsection{The formal reduction}
 
Fix an algebraic closure $\ol F$ of $F$ and a prime $\ell\nmid q$.  We
realize all finite-order character values in
$\overline{\mathbb Q}_\ell$.  For a curve $U/F$, the notation
$U_{\ol F}$ means base change to $\ol F$,
$H_c^i(U_{\ol F},\mathcal L)$ denotes compactly supported $\ell$-adic
cohomology, and $\mathcal L_{\ol x}$ is the stalk at a geometric point
$\ol x$ above $x$.  For the geometric picture, a finite-order rank-one
local system may be viewed as a rule assigning roots of unity to loops;
its local monodromy is the value on a small loop around a puncture.  We use
geometric Frobenius, denoted $\Frob_x$ locally and $\Frob_q$ over $F$.
We reserve this notation for geometric Frobenius; the $q$-power Frobenius
morphism used later in the Lang map is denoted $\Fr_q$.
 
The following lemma is the main technical estimate.  Its hypotheses
isolate the six-point symmetry needed in the application, while its proof
uses the classical trace formula and Weil bound cited below.

\begin{lemma}[Six-point quaternion bound]\label{lem:quaternion-reduction}
Let $C\simeq\mathbb P^1_F$, and let
$\mathcal L$ be a finite-order rank-one $\overline{\mathbb Q}_\ell$-local system on an open set $U\subset C$, of order prime to the
characteristic.  Suppose that, over $\ol F$,
\[
  C\setminus U=Z_0\sqcup Z_1\sqcup Z_2,
  \qquad |Z_i|=2,
\]
and that:
\begin{enumerate}[label=\textup{(\roman*)},leftmargin=2.6em]
\item a Klein four-group
\[
 V=\{1,\tau_0,\tau_1,\tau_2\}
\]
acts on $C_{\ol F}$, where $\tau_i$ fixes $Z_i$ and interchanges the two
points of each $Z_j$, $j\ne i$;
\item Frobenius cyclically permutes both triples $Z_0,Z_1,Z_2$ and
$\tau_0,\tau_1,\tau_2$;
\item after choosing compatible tame inertia generators (small loops around
the punctures), the local
monodromy at the two points of $Z_i$ is the same nonidentity scalar $a_i$,
and $a_0a_1a_2=-1$.
\end{enumerate}
Then
\[
 \left|
   \sum_{x\in U(F)}
   \tr\!\left(\Frob_x\mid\mathcal L_{\ol x}\right)
 \right|
 \le2\sqrt q.
\]
\end{lemma}
 
\begin{proof}
\medskip
\noindent\textit{Step 1: dimension and purity.}
Put $H=H_c^1(U_{\ol F},\mathcal L)$.  The Grothendieck trace formula
and the tame Euler--Poincar\'e formula give
\begin{equation}\label{eq:standard-curve-input}
 \sum_{x\in U(F)}
   \tr(\Frob_x\mid\mathcal L_{\ol x})
   =-\tr(\Frob_q\mid H),
 \qquad \dim H=6-2=4.
\end{equation}
Indeed, $H_c^0=0$ because $U$ is a connected nonproper curve,
and $H_c^2=0$ by Poincar\'e duality because $\mathcal L$ has
nontrivial geometric monodromy; all Swan conductors vanish by tameness.  The trace and Euler-characteristic
formulas are reviewed in \cite[Section~2.3]{KatzGKM}.

We also need purity, rather than only the general mixed-weight bound.
Trivialize $\mathcal L$ by the finite cyclic cover determined by its
character and let $X$ be the smooth projective completion.  Because the
local monodromy of $\mathcal L$ is nontrivial at every puncture, this
character has no inertia invariants above the boundary.  Hence
$H_c^1(U_{\ol F},\mathcal L)$ identifies with the corresponding character
summand of $H^1(X,\overline{\mathbb Q}_\ell)$.  The Weil theorem for the
smooth projective curve $X$ therefore shows that every Frobenius
eigenvalue on $H$ has complex absolute value $\sqrt q$ \cite{Weil1948}.
 
\medskip
\noindent\textit{Step 2: the projective $V$-action is genuinely nontrivial.}
The monodromy pattern in (i) and (iii) implies
$\tau_i^*\mathcal L\simeq\mathcal L$, because a tame rank-one system on a
punctured projective line is determined by its local monodromies.  Choose
isomorphisms $\varphi_v:v^*\mathcal L\to\mathcal L$ for $v\in V$.
Their failure to satisfy the group law is a scalar $2$-cocycle.  On
$H=H_c^1(U_{\ol F},\mathcal L)$, the corresponding operator is
\[
  T_v:
  H \xrightarrow{\,v^*\,}
  H_c^1(U_{\ol F},v^*\mathcal L)
  \xrightarrow{\,H_c^1(\varphi_v)\,}
  H.
\]
Thus $v\mapsto T_v$ is a projective representation with the same cocycle.

This cocycle is nontrivial.  Indeed, $V$ acts freely on $U$: the fixed points of the three nonidentity
involutions are precisely the six removed points.  The quotient of a
projective line by the finite group $V$ is again a projective line, and the
three $V$-orbits $Z_i$ give its three branch points.  Hence
\[
 U_{\ol F}/V\simeq
 \mathbb P^1_{\ol F}\setminus\{b_0,b_1,b_2\}.
\]
If the cocycle were a coboundary, the $\varphi_v$ could be rescaled to an
honest $V$-linearization.  Since $U\to U/V$ is finite \'etale, such a
linearization is precisely descent data, so $\mathcal L$ would descend to
this three-punctured line; see \cite[Chapter~I, \S2]{milne1980etale} for
the descent formalism.  If $r_i$ were its monodromy about $b_i$, then the compactified quotient
map $C_{\ol F}\to C_{\ol F}/V$ has ramification index two at the two
points of $Z_i$.  A tame loop about $b_i$ therefore lifts to the square
of a tame loop about either point of $Z_i$, so
\[
  r_i^2=a_i.
\]
On the three-punctured quotient one also has $r_0r_1r_2=1$.  Since these
monodromies are scalars, squaring would give
$a_0a_1a_2=1$, contradicting (iii).
 
\medskip
\noindent\textit{Step 3: quaternion algebra and tensor decomposition.}
Choose lifts $U_1,U_2$ of two independent nonidentity elements of
$V$.  For a Klein four-group over an algebraically closed field of
characteristic zero, the nontrivial projective class is detected by their
commutator: if the lifts commuted, the cocycle would be trivial.  After
rescaling the lifts, we may therefore arrange
\begin{equation}\label{eq:anticommuting-lifts}
 U_1^2=U_2^2=1,
 \qquad U_1U_2=-U_2U_1.
\end{equation}
Thus
\[
  1,\quad U_1,\quad U_2,\quad U_1U_2
\]
span the four-dimensional twisted group algebra with the standard
quaternionic matrix relations; hence this algebra is
$M_2(\overline{\mathbb Q}_\ell)$.  It has a unique irreducible module
$W$, of dimension two.  Since $H$ has dimension four, its
double-centralizer decomposition is
\begin{equation}\label{eq:tensor-decomposition}
  H\simeq W\otimes M,
  \qquad \dim W=\dim M=2,
\end{equation}
where the quaternionic matrices act on $W$ and $M$ is the multiplicity
space.
 
\medskip
\noindent\textit{Step 4: Frobenius on the quaternionic factor.}
Let $\mathcal A\simeq M_2(\overline{\mathbb Q}_\ell)$ denote the twisted group
algebra acting on $W$, and let $L_1,L_2,L_3\subset\mathcal A$ be the three
one-dimensional homogeneous lines corresponding to the nonidentity
elements of $V$.  By (ii), conjugation by Frobenius induces an algebra
automorphism $\alpha$ that cyclically permutes $L_1,L_2,L_3$.

Choose $I_i\in L_i$ with $I_i^2=-1$ and $I_1I_2=I_3$; then
$I_2I_3=I_1$ and $I_3I_1=I_2$.  Write
\[
  \alpha(I_1)=aI_2,\qquad
  \alpha(I_2)=bI_3,\qquad
  \alpha(I_3)=cI_1.
\]
Since $\alpha$ preserves squares, $a,b,c\in\{\pm1\}$, and applying
$\alpha$ to $I_1I_2=I_3$ gives $c=ab$.  Thus $abc=1$, so
$\alpha^3=1$ on the generators and hence on $\mathcal A$.  Since the
three lines are cyclically permuted, $\alpha$ is nontrivial and has exact
order three.

By the Skolem--Noether theorem, every automorphism of
$M_2(\overline{\mathbb Q}_\ell)$ is inner
\cite{Pierce}.  Hence $\alpha=\operatorname{Ad}(R_0)$ for some
$R_0\in\operatorname{GL}(W)$.  Since $\alpha^3=1$, the operator
$R_0^3$ is scalar.  We may first rescale $R_0$ so that $R_0^3=1$.
Its determinant is then a cube root of unity.  Multiplying once more by a
cube root of unity preserves $R_0^3=1$ and can be used to make the
determinant equal to $1$; call the resulting implementer $R$.  The ratio
of its two eigenvalues is a primitive cube root of unity, so the conditions
$R^3=1$ and $\det R=1$ force the eigenvalues to be the two primitive cube
roots.  Therefore
\begin{equation}\label{eq:R-trace}
  \tr(R)=-1.
\end{equation}
It follows that $R^{-1}\Frob_q$ centralizes $\mathcal A$.  By the
double-centralizer decomposition,
\begin{equation}\label{eq:frobenius-factorization}
  \Frob_q=R\otimes\Theta
\end{equation}
for some $\Theta\in\operatorname{End}(M)$.  The two eigenvalues of $R$ have
complex absolute value one, while all four eigenvalues of $\Frob_q$ on
$H$ have absolute value $\sqrt q$.  Hence the two eigenvalues of $\Theta$
have absolute value $\sqrt q$.

\medskip
\noindent\textit{Step 5: the trace bound.}
Therefore, using \eqref{eq:standard-curve-input},
\enlargethispage{\baselineskip}
\[
 \left|
   \sum_{x\in U(F)}
   \tr(\Frob_x\mid\mathcal L_{\ol x})
 \right|
 =|\tr(R)\tr(\Theta)|
 =|\tr(\Theta)|
 \le2\sqrt q.
\]
\end{proof}
 
\begin{remark}[Classical picture]\label{rem:qm-picture}
No descent of the individual
$\overline{\mathbb Q}_\ell$-character summand to an abelian subvariety is
needed.  The argument uses only the induced quaternionic action on
cohomology and the fact that Frobenius normalizes this action.
\end{remark}

\subsection{Application to the oriented conic}
\label{subsec:application-oriented-conic}

\begin{lemma}[Character system on the oriented conic]\label{lem:oriented-conic-sheaf}
Fix the orientation determined by $e=x_0$.  Choose $y\in E$ with
$B(e,y)=1$, set
\[
  f=y-\frac{Q(y)}{2}e,
\]
and choose a nonzero vector $h\in\langle e,f\rangle^\perp$.  Put
$\delta=Q(h)$ and
\begin{equation}\label{eq:veronese}
  P(u,v)=v^2f+uvh-\frac{\delta}{2}u^2e.
\end{equation}
Let $\lambda:E^\times\to\C^\times$ satisfy
$\lambda|_{F^\times}=\eta_F$.  Then the following hold.
\begin{enumerate}
\item The map $[u:v]\mapsto F^\times P(u,v)$ parametrizes $\mathcal C$, and
\[
  \mathcal C^+
  =\{F^{\times2}P(u,v):[u:v]\in\mathbb P^1(F)\}.
\]
\item Let $\mathbf T=\operatorname{Res}_{E/F}\mathbb G_m$, let
$\Fr_q:\mathbf T\to\mathbf T$ denote the $q$-power Frobenius morphism, and
let $\mathcal K_\lambda$ be the normalized Lang character sheaf associated
with $\lambda$.  Pulling $\mathcal K_\lambda$ back along the two standard
affine charts of \eqref{eq:veronese} gives systems that glue to a rank-one
local system $\mathcal L_\lambda$ on the open set $U\subset\mathbb P^1$
where the three geometric conjugates of $P(u,v)$ are nonzero.  Moreover,
$U(F)=\mathbb P^1(F)$ and
\begin{equation}\label{eq:S-as-curve-trace}
 S_\lambda
 =\sum_{[u:v]\in\mathbb P^1(F)}\lambda(P(u,v))
 =\sum_{x\in U(F)}
   \tr\!\left(\Frob_x\mid(\mathcal L_\lambda)_{\ol x}\right).
\end{equation}
\item Under the splitting
$\mathbf T_{\ol F}\simeq\mathbb G_m^3$, with split coordinates $x_i$ as
in \eqref{eq:split-conic}, one has locally near the boundary
\[
 (\mathcal L_\lambda)_{\ol F}
 \simeq
 \bigotimes_{i=0}^2\mathcal K_{\lambda^{q^i}}(x_i),
\]
where
\[
  \lambda^{q^i}(z):=\lambda(z^{q^i})=\lambda(z)^{q^i}.
\]
Consequently the tame local monodromy at either point of $Z_i$ is
$\chi_i=\lambda^{q^i}$.
\end{enumerate}
\end{lemma}

\begin{proof}
We have $Q(f)=0$ and $B(e,f)=1$, so the plane $\langle e,f\rangle$ is
nondegenerate.  Hence $(e,f,h)$ is an $F$-basis of $E$.  The Gram
determinant of this basis is $-\delta$, and
\eqref{eq:trace-discriminant-square} gives
\[
  -\delta\in F^{\times2}.
\]
A direct calculation gives $Q(P(u,v))=0$ and
\[
  B(e,P(u,v))=v^2.
\]
This is the standard degree-two parametrization of a nonsingular conic
from the rational point $[e]$, hence it identifies
$\mathbb P^1\simeq\mathcal C$.  For
$v\ne0$, the vector $P(u,v)$ itself represents the oriented lift.  At
$[1:0]$,
\[
  P(1,0)=-\frac{\delta}{2}e,
\]
and its ratio to $2e$ is $-\delta/4\in F^{\times2}$.  This proves (1).

For (2), the Lang map is
\[
  L:\mathbf T\longrightarrow\mathbf T,
  \qquad L(g)=g^{-1}\Fr_q(g).
\]
It is a finite \`etale Galois cover with Galois group
$\mathbf T(F)=E^\times$ \cite[Section~4.3]{KatzGKM}.  Composing the
resulting quotient of the \`etale fundamental group with $\lambda$ gives
the rank-one character sheaf $\mathcal K_\lambda$.  We use the standard
normalization for which, for $x\in\mathbf T(F)$,
\[
  \tr\!\left(\Frob_x\mid(\mathcal K_\lambda)_{\ol x}\right)=\lambda(x).
\]
With the opposite Frobenius convention one replaces $\lambda$ by
$\lambda^{-1}$, which does not affect any absolute-value estimate below.

The Lang character sheaf is multiplicative.  On the overlap of the two
standard charts,
\[
  P(t,1)=t^2P(1,t^{-1}).
\]
The restriction of $\mathcal K_\lambda$ to the scalar torus
$\mathbb G_m\subset\mathbf T$ is the quadratic Kummer system
$\mathcal K_{\eta_F}$, and its pullback by the squaring map
$[2]:\mathbb G_m\to\mathbb G_m$ is trivial.  Multiplicativity therefore
identifies the two chart pullbacks on the overlap, so they glue to
$\mathcal L_\lambda$ on $U$.  Since \eqref{eq:veronese} never gives the
zero vector on $\mathbb P^1(F)$, one has $U(F)=\mathbb P^1(F)$, and the
chosen normalization gives \eqref{eq:S-as-curve-trace}.  The equality at
$[1:0]$ follows from $P(1,0)=-(\delta/2)e$ and
$-\delta/4\in F^{\times2}$.

For (3), after base change to $\ol F$ the torus splits as
$\mathbb G_m^3$, and the Lang character sheaf decomposes as the tensor
product of the three Kummer systems associated with the Frobenius
conjugates $\lambda,\lambda^q,\lambda^{q^2}$.  Hence, on either standard
chart,
\[
 (\mathcal L_\lambda)_{\ol F}
 \simeq
 \bigotimes_{i=0}^2\mathcal K_{\lambda^{q^i}}(x_i).
\]
Near a point of $Z_i$, the function $x_i$ is a uniformizer while the other
two split coordinates are units.  Multiplying the homogeneous
representative by a unit changes only those unit factors and hence does
not change tame inertia.  Thus the local inertia character at $Z_i$ is
exactly $\lambda^{q^i}$.
\end{proof}

\begin{proof}[Proof of \Cref{thm:period-bound}]
Apply \Cref{lem:oriented-conic-sheaf}.  The three pairs $Z_i$ of
\eqref{eq:coordinate-pairs}, the coordinate sign changes, and the cyclic
geometric Frobenius action satisfy conditions \textup{(i)} and \textup{(ii)}
of \Cref{lem:quaternion-reduction}.  By part \textup{(3)} of the lemma,
the local monodromies are
\[
  \chi_i=\lambda^{q^i},\qquad i=0,1,2,
\]
and all three are nontrivial.  Moreover,
\begin{equation}\label{eq:monodromy-product-minus-one}
  \chi_0\chi_1\chi_2
  =\lambda^{1+q+q^2}
  =\lambda^N
  =\eta_E,
\end{equation}
where $\eta_E$ is the quadratic character of $E^\times$.  Indeed, if
$\operatorname{N}_{E/F}$ denotes the field norm, then for $z\in E^\times$,
\[
  \lambda^N(z)=\lambda(\operatorname{N}_{E/F}z)
  =\eta_F(\operatorname{N}_{E/F}z)=\eta_E(z).
\]
Use the local parameters $x_i$ to identify the tame inertia groups and
choose a common generator $\gamma$.  If $a_i=\chi_i(\gamma)$, then the
quadratic tame character takes $\gamma$ to $-1$, so
\eqref{eq:monodromy-product-minus-one} gives
$a_0a_1a_2=-1$.  All hypotheses of \Cref{lem:quaternion-reduction} hold,
and \eqref{eq:S-as-curve-trace} gives
\[
  |S_\lambda|\le2\sqrt q.
\]
As noted after \Cref{prop:conic-orientations}, replacing $\mathcal C^+$ by
the other coherent orientation changes $S_\lambda$ only by a sign, so the
same bound holds for either orientation.
\end{proof}

\begin{proof}[Proof of \Cref{thm:main}]
\Cref{prop:lift-cover} gives the regular bipartite double cover of the
intrinsic Cayley model in \eqref{eq:main-graph}.  By
\Cref{prop:base-cayley}, the old nontrivial eigenvalues have magnitude
$\sqrt q$, and the new eigenvalues
have magnitude at most $2\sqrt q$ by \Cref{thm:period-bound}.  Hence the
cover satisfies the Ramanujan bound.  Since $2\sqrt q<q+1$, no new
eigenvalue equals $q+1$;
the multiplicity of the top eigenvalue is therefore one, so the cover is
connected.  Consequently $\ell_Q$ is a Ramanujan section by
Definition~\ref{def:ramanujan-section}.
\end{proof}

\section{Characteristic two}
\label{sec:characteristic-two}

The odd-characteristic construction does not extend directly to even
$q$.  Indeed, the square-class quotient $F^\times/F^{\times2}$ is
trivial, and
\[
  Q(x)=\Tr(x^2)=\Tr(x)^2
\]
defines a nonreduced double line rather than a nonsingular conic.
Consequently, both the oriented-conic construction and the six-point
symmetry degenerate.

One may instead sign the usual Singer set
$D_0\subseteq\Z/N\Z$.  For signs
$\varepsilon_d\in\{\pm1\}$, put
\[
  T_k(\varepsilon)
  =
  \sum_{d\in D_0}\varepsilon_d\zeta_N^{kd},
  \qquad
  \zeta_N=e^{2\pi i/N},
\]
and impose the natural Frobenius symmetry
\[
  \varepsilon_{qd}=\varepsilon_d.
\]

A small computation already shows that this substitute need not
produce a Ramanujan cover.  For $q=32$, one has $N=1057$, and $D_0$ is
a union of eleven Frobenius orbits.  Exhaustive enumeration of the
$2^{11}$ invariant signings gives
\[
  \min_{\varepsilon_{32d}=\varepsilon_d}
  \max_{k\in\Z/1057\Z}|T_k(\varepsilon)|
  =
  11.462986587830\ldots
  >
  2\sqrt{32}.
\]
Thus no Frobenius-stable dihedral Cayley signing is Ramanujan in this
case.

This observation concerns only the symmetry-restricted Cayley problem;
it does not contradict the MSS theorem, which still guarantees a
Ramanujan double cover without this restriction.

\section*{Declaration of generative AI and AI-assisted technologies in the manuscript preparation process}
During the preparation of this work, the authors used ChatGPT to assist with exposition and manuscript editing.  After using these tools, the
authors reviewed and edited the content as needed and take full
responsibility for the content of the publication.

\begingroup
\small
\bibliographystyle{amsplain}
\bibliography{ramanujan_double_covers_odd_q}
\endgroup

\end{document}